\documentclass[11pt]{amsart}
\usepackage[margin=30mm]{geometry}
\usepackage{amsmath,amssymb}
\usepackage{amsthm}
\usepackage{mathrsfs}

\newtheorem{thm}{Theorem}[section]

\newtheorem{lem}{Lemma}[section]
\newtheorem{cor}{Corollary}[section]
\newtheorem{defi}{Definition}[section]

\newtheorem{rem}{Remark}[section]

\newtheorem*{theorem}{\it{Theorem}}

\usepackage{enumitem}

\begin{document}

\title{S-limit shadowing and a global description of Li--Yorke type chaos}
\author{Noriaki Kawaguchi}
\address{Research Institute of Science and Technology, Tokai University, 4-1-1 Kitakaname, Hiratsuka, Kanagawa 259-1292, Japan}
\email{gknoriaki@gmail.com}

\begin{abstract}
For any continuous self-map of a compact metric space, we extend a partition of each chain component with respect to a chain proximal relation to a $G_\delta$-partition of the phase space. Under the assumption of s-limit shadowing, we use this partition to give a global description of Li--Yorke type chaos corresponding to several Furstenberg families.   
\end{abstract}

\maketitle

\markboth{NORIAKI KAWAGUCHI}{S-limit shadowing and a global description of Li--Yorke type chaos}

\section{Introduction}

{\em Shadowing} is a hyperbolic feature of topological dynamical systems and has played an important role in the global theory of dynamical systems (see \cite{AH,P} for background). It was introduced in the works of Anosov and Bowen \cite{An,B}; and generally refers to the phenomenon in which coarse orbits, or {\em pseudo-orbits}, are approximated by true orbits. Various types of shadowing properties have been a subject of intensive research (see, for example, \cite{WOC}). As an expression of hyperbolicity of dynamical systems, shadowing is closely related to the complex phenomena known as `chaos'.

The coarse structure of pseudo-orbits, or {\em chains}, is of independent interest. The chain components as they appear in the so-called fundamental theorem of dynamical systems by Conley \cite{C} are basic objects in the theory of dynamical systems.  Given a compact metric space $X$ and a continuous map
\[
f\colon X\to X,
\]
let $\mathcal{C}(f)$ denote the set of chain components for $f$. The basin $W^s(C)$ of $C$ is a $G_\delta$-subset of $X$ for every $C\in\mathcal{C}(f)$. Then, $X$ is partitioned into the disjoint union of the basins of chain components:
\[
X=\bigsqcup_{C\in\mathcal{C}(f)}W^s(C).
\]
A {\em chain proximal relation} on a chain component was introduced in \cite{S} (see also \cite{RW}). In this paper, we consider a partition $\mathcal{D}(C)$ of each $C\in\mathcal{C}(f)$ with respect to the chain proximal relation:
\[
C=\bigsqcup_{D\in\mathcal{D}(C)}D. 
\]
By considering a kind of stable set $V^s(D)$ for each $D\in\mathcal{D}(C)$, we introduce a partition of $W^s(C)$:
\[
W^s(C)=\bigsqcup_{D\in\mathcal{D}(C)}V^s(D)
\]
such that $V^s(D)$ is a $G_\delta$-subset of $W^s(C)$ for every $D\in\mathcal{D}(C)$. Accordingly, we obtain a partition of $X$:
\[
X=\bigsqcup_{C\in\mathcal{C}(f)}W^s(C)=\bigsqcup_{C\in\mathcal{C}(f)}\bigsqcup_{D\in\mathcal{D}(C)}V^s(D)
\]
into $G_\delta$-subsets of $X$. We use this partition to give a global description of Li--Yorke type chaos.

The {\em  Li--Yorke chaos} introduced in \cite{LY2} is a standard rigorous definition of chaos.  In this paper, we consider a few variations of Li--Yorke chaos with respect to certain {\em Furstenberg families}. Under the assumption of a variation of shadowing, namely {\em s-limit shadowing}, we describe a global chaotic structure in terms of Li--Yorke type chaos. For any $x_1,x_2,\dots, x_n\in X$ and $r>0$, let
\[
S_f(x_1,x_2,\dots,x_n;r)=\{i\in\mathbb{N}_0\colon\min_{1\le j<k\le n}d(f^i(x_j),f^i(x_k))>r\}
\] 
and
\[
T_f(x_1,x_2,\dots,x_n;r)=\{i\in\mathbb{N}_0\colon\max_{1\le j<k\le n}d(f^i(x_j),f^i(x_k))<r\},
\]
where $d$ is a metric on $X$. We consider four Furstenberg families
\[
\mathcal{F}_{\rm ud1},\mathcal{F}_{\rm thick},\mathcal{F}_{\rm iap}^\ast,\mathcal{F}_{\rm infinite}\subset2^{\mathbb{N}_0}.
\]
The main results are the following three theorems. A {\em $G_\delta$-subset} of a topological space $Z$ is a subset of $Z$ that is a countable intersection of open subsets of $Z$.

\begin{thm}
Let $f\colon X\to X$ be a continuous map and let $C\in\mathcal{C}(f)$. If $f$ has the s-limit shadowing property, then for any $n\ge2$, the following conditions are equivalent
\begin{description}
\item[\rm$\:\:$(1)] There are $D\in\mathcal{D}(C)$ and $\delta_n>0$ such that $D^n$ contains a $\delta_n$-distal $n$-tuple for $f$.
\item[\rm$\:\:$(1')] There is $\delta_n>0$ such that for any $D\in\mathcal{D}(C)$, $D^n$ contains a $\delta_n$-distal $n$-tuple for $f$.
\item[\rm$\:\:$(2)] There is $\delta_n>0$ such that for every $D\in\mathcal{D}(C)$,
\[
\{(x_1,x_2,\dots,x_n)\in[V^s(D)]^n\colon S_f(x_1,x_2,\dots,x_n;\delta_n)\in\mathcal{F}_{\rm ud1}\}
\]
and
\[
\bigcap_{\epsilon>0}\{(x_1,x_2,\dots,x_n)\in[V^s(D)]^n\colon T_f(x_1,x_2,\dots,x_n;\epsilon)\in\mathcal{F}_{\rm ud1}\}
\]
are dense $G_\delta$-subsets of $[V^s(D)]^n$.
\item[\rm$\:\:\:$(3)] There are $(y_1,y_2,\dots,y_n)\in[W^s(C)]^n$ and $\delta_n>0$ such that
\[
S_f(y_1,y_2,\dots,y_n;\delta_n)\in\mathcal{F}_{\rm thick},
\]
and
\[
T_f(y_1,y_2,\dots,y_n;\epsilon)\in\mathcal{F}_{\rm infinite}
\]
for all $\epsilon>0$.
\end{description}
\end{thm}

\begin{thm}
Let $f\colon X\to X$ be a continuous map and let $C\in\mathcal{C}(f)$. If $f$ has the s-limit shadowing property, then for any $n\ge2$, the following conditions are equivalent
\begin{description}
\item[\rm$\:\:$(1)]
\[
\Delta_n=\inf_{D\in\mathcal{D}(C)}\sup_{x_1,x_2,\dots,x_n\in D}\min_{1\le j<k\le n}d(x_j,x_k)>0.
\]
\item[\rm$\:\:$(2)] There is $\delta_n>0$ such that for every $D\in\mathcal{D}(C)$,
\[
\{(x_1,x_2,\dots,x_n)\in[V^s(D)]^n\colon S_f(x_1,x_2,\dots,x_n;\delta_n)\in\mathcal{F}_{\rm iap}^\ast\}
\]
and
\[
\bigcap_{\epsilon>0}\{(x_1,x_2,\dots,x_n)\in[V^s(D)]^n\colon T_f(x_1,x_2,\dots,x_n;\epsilon)\in\mathcal{F}_{\rm ud1}\}
\]
are dense $G_\delta$-subsets of $[V^s(D)]^n$.
\item[\rm$\:\:$(3)] There are $(y_1,y_2,\dots,y_n)\in[W^s(C)]^n$ and $\delta_n>0$ such that
\[
S_f(y_1,y_2,\dots,y_n;\delta_n)\in\mathcal{F}_{\rm iap}^\ast,
\]
and
\[
T_f(y_1,y_2,\dots,y_n;\epsilon)\in\mathcal{F}_{\rm infinite}
\]
for all $\epsilon>0$.
\end{description}
\end{thm}

\begin{thm}
Let $f\colon X\to X$ be a continuous map and let $C\in\mathcal{C}(f)$. If $f$ has the s-limit shadowing property, then for any $n\ge2$, the following conditions are equivalent
\begin{description}
\item[\rm$\:\:$(1)] There is $D\in\mathcal{D}(C)$ such that $|D|\ge n$.
\item[\rm$\:\:$(2)] There is $\delta_n>0$ such that for every $D\in\mathcal{D}(C)$,
\[
\{(x_1,x_2,\dots,x_n)\in[V^s(D)]^n\colon S_f(x_1,x_2,\dots,x_n;\delta_n)\in\mathcal{F}_{\rm infinite}\}
\]
and
\[
\bigcap_{\epsilon>0}\{(x_1,x_2,\dots,x_n)\in[V^s(D)]^n\colon T_f(x_1,x_2,\dots,x_n;\epsilon)\in\mathcal{F}_{\rm ud1}\}
\]
are dense $G_\delta$-subsets of $[V^s(D)]^n$.
\item[\rm$\:\:$(3)] There are $(y_1,y_2,\dots,y_n)\in[W^s(C)]^n$ and $\delta_n>0$ such that
\[
S_f(y_1,y_2,\dots,y_n;\delta_n)\in\mathcal{F}_{\rm infinite},
\]
and
\[
T_f(y_1,y_2,\dots,y_n;\epsilon)\in\mathcal{F}_{\rm infinite}
\]
for all $\epsilon>0$.
\end{description}
\end{thm}

The three theorems have the similar structure. For each $C\in\mathcal{C}(f)$, the first condition (1) (or (1')) describes a property of $C$. The second condition (2) describes a chaotic structure in $W^s(C)$. The third condition (3) is weaker than (2) and in fact implies (2).  Since
\[
\mathcal{F}_{\rm ud1}\subset\mathcal{F}_{\rm thick}\subset\mathcal{F}_{\rm iap}^\ast\subset\mathcal{F}_{\rm infinite},
\]
the three types of Li--Yorke chaos considered in Theorems 1.1, 1.2, and 1.3 are in order of strength. We remark that Li--Yorke type chaos with respect to $\mathcal{F}_{\rm ud1}$ (resp.\:$\mathcal{F}_{\rm infinite}$) are called {\em distributional chaos of type 1} \cite{SS} (resp.\:Li--Yorke chaos). Under the assumption of s-limit shadowing, the three theorems provide a detailed description of the chaotic structure in the basins of chain components corresponding to the structure of each chain component. As a consequence, we introduce a viewpoint of classifying (in general, uncountably many) chain components by strength of Li--Yorke type chaos in the basins of them.

This paper consists of six sections. Throughout, $X$ denotes a compact metric space endowed with a metric $d$. In Section 2, we introduce the $G_\delta$-partition of the phase space. In Section 3, we introduce Furstenberg families and corresponding Li--Yorke type chaos. In Section 4, we define the s-limit shadowing and prove a lemma. In Section 5, we prove the main theorems. Finally, in Section 6, we make some concluding remarks. 

\section{A $G_\delta$-partition of the phase space}

In this section, we introduce the $G_\delta$-partition of the phase space mentioned in Section 1. We begin with a definition.

\begin{defi}
\normalfont
Given a continuous map $f\colon X\to X$ and $\delta>0$, a finite sequence $(x_i)_{i=0}^{k}$ of points in $X$, where $k>0$ is a positive integer, is called a {\em $\delta$-chain} of $f$ if $d(f(x_i),x_{i+1})\le\delta$ for every $0\le i\le k-1$.  A $\delta$-chain $(x_i)_{i=0}^{k}$ of $f$ with $x_0=x_k$ is said to be a {\em $\delta$-cycle} of $f$. 
\end{defi}

Let $f\colon X\to X$ be a continuous map. For any $x,y\in X$ and $\delta>0$, the notation $x\rightarrow_\delta y$ means that there is a $\delta$-chain $(x_i)_{i=0}^k$ of $f$ with $x_0=x$ and $x_k=y$. We write $x\rightarrow y$ if $x\rightarrow_\delta y$ for all $\delta>0$. We say that $x\in X$ is a {\em chain recurrent point} for $f$ if $x\rightarrow x$, or equivalently, for any $\delta>0$, there is a $\delta$-cycle $(x_i)_{i=0}^{k}$ of $f$ with $x_0=x_k=x$. Let $CR(f)$ denote the set of chain recurrent points for $f$. We define a relation $\leftrightarrow$ in
\[
CR(f)^2=CR(f)\times CR(f)
\]
by: for any $x,y\in CR(f)$, $x\leftrightarrow y$ if and only if $x\rightarrow y$ and $y\rightarrow x$. Note that $\leftrightarrow$ is a closed equivalence relation in $CR(f)^2$ and satisfies $x\leftrightarrow f(x)$ for all $x\in CR(f)$. An equivalence class $C$ of $\leftrightarrow$ is called a {\em chain component} for $f$. We denote by $\mathcal{C}(f)$ the set of chain components for $f$.

A subset $S$ of $X$ is said to be $f$-invariant if $f(S)\subset S$. For an $f$-invariant subset $S$ of $X$, we say that $f|_S\colon S\to S$ is {\em chain transitive} if for any $x,y\in S$ and $\delta>0$, there is a $\delta$-chain $(x_i)_{i=0}^k$ of $f|_S$ with $x_0=x$ and $x_k=y$.

\begin{rem}
\normalfont
The following properties hold
\begin{itemize}
\item $CR(f)=\bigsqcup_{C\in\mathcal{C}(f)}C$,
\item Every $C\in\mathcal{C}(f)$ is a closed $f$-invariant subset of $CR(f)$,
\item $f|_C\colon C\to C$ is chain transitive for all $C\in\mathcal{C}(f)$,
\item For any $f$-invariant subset $S$ of $X$, if $f|_S\colon S\to S$ is chain transitive, then $S\subset C$ for some $C\in\mathcal{C}(f)$.
\end{itemize}
\end{rem}

\begin{rem}
\normalfont
We recall the so-called fundamental theorem of dynamical systems by Conley \cite{C} (see also \cite{H}). Given any continuous map $f\colon X\to X$, we say that a continuous function $\Lambda\colon X\to\mathbb{R}$ is a {\em complete Lyapunov function} for $f$ if the following conditions are satisfied
\begin{itemize}
\item $\Lambda(f(x))<\Lambda(x)$ for every $x\in X\setminus CR(f)$,
\item For any $x,y\in CR(f)$,  $\Lambda(x)=\Lambda(y)$ if and only if $x\leftrightarrow y$,
\item $\Lambda(CR(f))$ is a nowhere dense subset of $\mathbb{R}$.
\end{itemize}

\begin{theorem}
There is a complete Lyapunov function $\Lambda\colon X\to\mathbb{R}$ for $f$.
\end{theorem}
\end{rem}

Let $C\in\mathcal{C}(f)$ and let $g=f|_C\colon C\to C$. Given any $\delta>0$, the {\em length} of a $\delta$-cycle $(x_i)_{i=0}^k$ of $g$ is defined to be $k$. Let $m=m(C,\delta)>0$ be the greatest common divisor of the lengths of all $\delta$-cycles of $g$. A relation $\sim_{C,\delta}$ in $C^2$ is defined by: for any $x,y\in C$, $x\sim_{C,\delta}y$ if and only if there is a $\delta$-chain $(x_i)_{i=0}^k$ of $g$ with $x_0=x$, $x_k=y$, and $m|k$. 

\begin{rem}
\normalfont
The following properties hold 
\begin{itemize}
\item[(P1)] $\sim_{C,\delta}$ is an open and closed $(g\times g)$-invariant equivalence relation in $C^2=C\times C$,
\item[(P2)] Any $x,y\in C$ with $d(x,y)\le\delta$ satisfies $x\sim_{C,\delta}y$, so for every $\delta$-chain $(x_i)_{i=0}^k$ of $g$, we have $g(x_i)\sim_{C,\delta}x_{i+1}$ for each $0\le i\le k-1$, implying $x_i\sim_{C,\delta}g^i(x_0)$ for every $0\le i\le k$,
\item[(P3)] For any $x\in C$ and $n\ge0$, $x\sim_{C,\delta}g^{mn}(x)$,
\item[(P4)] There exists $N>0$ such that for any $x,y\in C$ with $x\sim_{C,\delta}y$ and $n\ge N$, there is a $\delta$-chain $(x_i)_{i=0}^k$ of $g$ with $x_0=x$, $x_k=y$, and $k=mn$.
\end{itemize}
A proof of property (P2) can be found in \cite{K3}. Property (P4) is due to Lemma 2.3 of \cite{BMR} and in fact a consequence of a Schur's theorem implying that for any positive integers $m,k_1,k_2,\dots,k_n$, if
\[
\gcd(k_1,k_2,\dots,k_n)=m,
\]
then every sufficiently large multiple $M$ of $m$ can be expressed as a linear combination
\[
M=a_1k_1+a_2k_2+\cdots+a_nk_n
\]
where $a_1,a_2,\dots,a_n$ are non-negative integers. 
\end{rem}

Fix $x\in C$ and let $D_i$, $i\ge0$, denote the equivalence class of $\sim_{C,\delta}$ including $g^i(x)$. Then, $D_m=D_0$, and
\[
C=\bigsqcup_{i=0}^{m-1}D_i
\]
gives the partition of $C$ into the equivalence classes of $\sim_{C,\delta}$. Note that every $D_i$, $0\le i\le m-1$, is an open and closed subset of $C$ and satisfies $g(D_i)=D_{i+1}$. We call
\[
\mathcal{D}(C,\delta)=\{D_i\colon 0\le i\le m-1\}
\]
the {\em $\delta$-cyclic decomposition} of $C$.

\begin{defi}
\normalfont
We define a relation $\sim_C$ in $C^2$ by: for any $x,y\in C$, $x\sim_C y$ if and only if $x\sim_{C,\delta}y$ for all $\delta>0$, which is a closed $(g\times g)$-invariant equivalence relation in $C^2$. We denote by $\mathcal{D}(C)$ the set of equivalence classes of $\sim_C$.
\end{defi}

\begin{rem}
\normalfont
The relation $\sim_C$ was introduced in \cite{S} and rediscovered in \cite{RW} (based on the argument given in \cite[Exercise 8.22]{A}). We say that $(x,y)\in C^2$ is a {\em chain proximal pair} for $g$ if for any $\delta>0$, there is a pair of $\delta$-chains
\[
((x_i)_{i=0}^k,(y_i)_{i=0}^k)
\]
of $g$ with $(x_0,y_0)=(x,y)$ and $x_k=y_k$. As stated in Remark 8 of \cite{RW}, it holds that for any $x,y\in C$, $x\sim_C y$ if and only if $(x,y)$ is a chain proximal pair for $g$. By this equivalence, we say that $\sim_C$ is a {\em chain proximal relation} for $g$.
\end{rem}

Given any continuous map $f\colon X\to X$, let
\[
W^s(C)=\{x\in X\colon\lim_{i\to\infty}d(f^i(x),C)=0\}
\]
for all $C\in\mathcal{C}(f)$.  Note that $C\subset W^s(C)$ for every $C\in\mathcal{C}(f)$.

\begin{lem}
We have the following properties
\begin{itemize}
\item
\[
X=\bigsqcup_{C\in\mathcal{C}(f)}W^s(C),
\]
\item $W^s(C)$ is a $G_\delta$-subset of $X$ for every $C\in\mathcal{C}(f)$.
\end{itemize}
\end{lem}

\begin{proof}
For $x\in X$, the {\em $\omega$-limit set} $\omega(x,f)$ of $x$ for $f$ is defined to be the set of $y\in X$ such that $\lim_{j\to\infty}f^{i_j}(x)=y$ for some sequence $0\le i_1<i_2<\cdots$. Note that $\omega(x,f)$ is a closed $f$-invariant subset of $X$ and satisfies
\[
\lim_{i\to\infty}d(f^i(x),\omega(x,f))=0.
\]
Since $f|_{\omega(x,f)}\colon\omega(x,f)\to\omega(x,f)$ is chain transitive, $\omega(x,f)\subset C$ for some $C\in\mathcal{C}(f)$. Then, for any $C\in\mathcal{C}(f)$ and $x\in X$, $x\in W^s(C)$ if and only if $\omega(x,f)\subset C$. It follows that
\[
X=\bigsqcup_{C\in\mathcal{C}(f)}W^s(C).
\]
For any non-empty closed subsets $A,B$ of $X$, let
\[
d(A,B)=\min_{x\in A,y\in B}d(x,y).
\]
For any $C,C'\in\mathcal{C}(f)$ and $x\in X$, if $x\in W^s(C')$ and $C\ne C'$, then
\[
\liminf_{i\to\infty}d(f^i(x),C)\ge d(C,C')>0,
\]
so $x\in W^s(C)$ if and only if
\[
\liminf_{i\to\infty}d(f^i(x),C)=0.
\]
We obtain
\[
W^s(C)=\{x\in X\colon\liminf_{i\to\infty}d(f^i(x),C)=0\}=\bigcap_{l\ge1}\bigcap_{q\ge0}\bigcup_{i\ge q}\{x\in X\colon d(f^i(x),C)<\frac{1}{l}\};
\]
therefore, $W^s(C)$ is a $G_\delta$-subset of $X$, proving the lemma.
\end{proof}

Let $C\in\mathcal{C}(f)$ and let $D\in\mathcal{D}(C)$. By the definition of $\sim_C$, for any $\delta>0$, there is unique $D_\delta\in\mathcal{D}(C,\delta)$ such that $D\subset D_\delta$. Then, $D_{\delta_1}\subset D_{\delta_2}$ for all $0<\delta_1<\delta_2$, and
\[
D=\bigcap_{\delta>0}D_\delta.
\]
Let
\[
V^s(D)=\bigcap_{\delta>0}\{x\in W^s(C)\colon\lim_{i\to\infty}d(f^i(x),f^i(D_{\delta}))=0\}
\]
for all $D\in\mathcal{D}(C)$. Note that $D\subset V^s(D)$ for every $D\in\mathcal{D}(C)$.

\begin{lem}
We have the following properties
\begin{itemize}
\item
\[
W^s(C)=\bigsqcup_{D\in\mathcal{D}(C)}V^s(D),
\]
\item $V^s(D)$ is a $G_\delta$-subset of $W^s(C)$ for every $D\in\mathcal{D}(C)$.
\end{itemize}
\end{lem}

\begin{proof}
For any $D,D'\in\mathcal{D}(C)$, $x\in V^s(D)$, and $y\in V^s(D')$, if $D\ne D'$, then by taking $\delta>0$ with $D_\delta\ne D'_\delta$, we have
\[
\liminf_{i\to\infty}d(f^i(x),f^i(y))\ge\inf_{i\ge0}d(f^i(D_\delta),f^i(D'_\delta))>0.
\]
This implies $V^s(D)\cap V^s(D')=\emptyset$ for all $D,D'\in\mathcal{D}(C)$ with $D\ne D'$. Given any $x\in W^s(C)$ and $\delta>0$, we have
\[
\lim_{i\to\infty}d(f^i(x),D_{\delta}(x))=0
\]
for unique $D_{\delta}(x)\in\mathcal{D}(C,\delta)$. Since $D_{\delta_1}(x)\subset D_{\delta_2}(x)$ for all $0<\delta_1<\delta_2$, letting
\[
D(x)=\bigcap_{\delta>0}D_{\delta}(x),
\]
we obtain $D(x)\in\mathcal{D}(C)$ and $x\in V^s(D(x))$. Since $x\in W^s(C)$ is arbitrary, it follows that
\[
W^s(C)=\bigsqcup_{D\in\mathcal{D}(C)}V^s(D).
\]
For any $D,D'\in\mathcal{D}(C)$ and $x\in V^s(D')$, if $D\ne D'$, then by taking $\delta>0$ with $D_\delta\ne D'_\delta$, we have
\[
\liminf_{i\to\infty}d(f^i(x),f^i(D_\delta))\ge\inf_{i\ge0}d(f^i(D_\delta),f^i(D'_\delta))>0.
\]
It follows that for any $x\in W^s(C)$ and $D\in\mathcal{D}(C)$, $x\in V^s(D)$ if and only if
\[
\liminf_{i\to\infty}d(f^i(x),f^i(D_\delta))=0
\]
for all $\delta>0$. We obtain
\begin{align*}
V^s(D)&=\bigcap_{\delta>0}\{x\in W^s(C)\colon\liminf_{i\to\infty}d(f^i(x),f^i(D_\delta))=0\}\\
&=\bigcap_{n\ge1}\bigcap_{l\ge1}\bigcap_{q\ge0}\bigcup_{i\ge q}\{x\in W^s(C)\colon d(f^i(x),f^i(D_\frac{1}{n}))<\frac{1}{l}\};
\end{align*}
therefore, $V^s(D)$ is a $G_\delta$-subset of $W^s(C)$, proving the lemma.
\end{proof}

\begin{rem}
\normalfont
We can show that for any $C\in\mathcal{C}(f)$ and $D\in\mathcal{D}(C)$,
\[
V^s(D)=\{x\in W^s(C)\colon\liminf_{i\to\infty}d(f^i(x),f^i(D))=0\}.
\]
\end{rem}

By the above two lemmas, we obtain
\[
X=\bigsqcup_{C\in\mathcal{C}(f)}W^s(C)=\bigsqcup_{C\in\mathcal{C}(f)}\bigsqcup_{D\in\mathcal{D}(C)}V^s(D).
\]
Moreover, $V^s(D)$ is a $G_\delta$-subset of $X$ for all $C\in\mathcal{C}(f)$ and $D\in\mathcal{D}(C)$.

\section{Furstenberg families and chaos}

In this section, we introduce Furstenberg families and corresponding Li--Yorke type chaos. Let $\mathbb{N}_0=\{0\}\cup\mathbb{N}=\{0,1,2,\dots\}$ and let $\mathcal{F}\subset2^{\mathbb{N}_0}$. We say that
$\mathcal{F}$ is a {\em Furstenberg family} if the following conditions are satisfied
\begin{itemize}
\item (hereditary upward) For any $A,B\subset\mathbb{N}_0$, $A\in\mathcal{F}$ and $A\subset B$ implies $B\in\mathcal{F}$,
\item (proper) $\mathcal{F}\ne\emptyset$ and $\mathcal{F}\ne2^{\mathbb{N}_0}$.
\end{itemize}

For any Furstenberg family $\mathcal{F}$, we define its dual family $\mathcal{F}^\ast$ by
\[
\mathcal{F}^\ast=\{A\subset\mathbb{N}_0\colon A\cap B\ne\emptyset\:\:\text{for all $B\in\mathcal{F}$}\},
\]
which is also a Furstenberg family.
\begin{rem}
\normalfont
The idea of combining families of subsets of $\mathbb{N}_0$ with dynamical properties of maps was first used by Gottschalk and Hedlund \cite{GH}; and then further developed by Furstenberg \cite{F} (see \cite{LY1} and references therein).
\end{rem}

We define the four Furstenberg families: $\mathcal{F}_{\rm ud1}$, $\mathcal{F}_{\rm thick}$, $\mathcal{F}_{\rm iap}$, $\mathcal{F}_{\rm infinite}$.

\begin{itemize}
\item For any $A\subset\mathbb{N}_0$, we denote by $\overline{d}(A)$ the {\em upper density} of $A$:
\[
\overline{d}(A)=\limsup_{n\to\infty}\frac{1}{n}|A\cap\{0,1,\dots,n-1\}|,
\]
here $|\cdot|$ is the cardinality of the set.
\[
\mathcal{F}_{\rm ud1}=\{A\subset\mathbb{N}_0\colon\overline{d}(A)=1\}.
\]
\item
\[
\mathcal{F}_{\rm thick}=\{A\subset\mathbb{N}_0\colon\{i_j,i_j+1,\dots,i_j+j-1\}\subset A\:\:\text{for all $j\ge1$ for some $i_j\ge0$}\}.
\]
\item For integers $p\ge0$ and $m\ge1$, we define $\langle p,m\rangle$ by
\[
\langle p,m\rangle=\{p+qm\colon q\ge0\}=\{p, p+m,p+2m,\dots\}.
\]
\[
\mathcal{F}_{\rm iap}=\{A\subset\mathbb{N}_0\colon\langle p,m\rangle\subset A\:\:\text{for some $p\ge0$ and some $m\ge1$}\}.
\]
\item
\[
\mathcal{F}_{\rm infinite}=\{A\subset\mathbb{N}_0\colon |A|=\infty\}.
\]
\end{itemize}

\begin{rem}
\normalfont
\begin{itemize}
\item[(1)]Note that
\[
\mathcal{F}_{\rm ud1}\subset\mathcal{F}_{\rm thick}\subset\mathcal{F}_{\rm iap}^\ast\subset\mathcal{F}_{\rm infinite}.
\]
\item[(2)] Let
\[
X=S^1=\{z\in\mathbb{C}\colon|z|=1\}.
\]
Fix an irrational number $\alpha>0$ and let
\[
f=R_\alpha\colon X\to X
\]
be the irrational rotation defined by
$f(z)=ze^{2\pi i\alpha}$ for all $z\in X$. We let
\[
U=\{z\in X\colon\Re z<0\}
\]
and
\[
A=\{i\in\mathbb{N}_0\colon f^i(1)\in U\}.
\]
For any $p\ge0$ and $m\ge1$, since $f^m\colon X\to X$ is minimal, we have
\[
f^{p+qm}(1)=f^{qm}(f^p(1))\in U
\]
for some $q\ge0$. This implies $A\in\mathcal{F}_{\rm iap}^\ast$. On the other hand, since $f\colon X\to X$ is minimal, we know that for any $x\in X$ and a non-empty open subset $V$ of $X$,
\[
\{i\in\mathbb{N}_0\colon f^i(x)\in V\}
\]
has bounded gaps, that is, there is $k\ge1$ such that
\[
\{i\in\mathbb{N}_0\colon f^i(x)\in V\}\cap\{j,j+1,\dots,j+k-1\}\ne\emptyset
\]
for all $j\ge0$. Letting $x=1$ and
\[
V=\{z\in X\colon\Re z>0\},
\]
we see that $A\not\in\mathcal{F}_{\rm thick}$, thus $\mathcal{F}_{\rm iap}^\ast\setminus\mathcal{F}_{\rm thick}\ne\emptyset$.
\end{itemize}
\end{rem}

We consider Li--Yorke type chaos corresponding to the above four Furstenberg families. The  notion of chaos via Furstenberg families was introduced in \cite{TX,XLT} (see also \cite{LY1}). Let $f\colon X\to X$ be a continuous map. For any $n\ge2$ and $x_1,x_2,\dots, x_n\in X$, we define
\[
S_f(x_1,x_2,\dots,x_n;r)=\{i\in\mathbb{N}_0\colon\min_{1\le j<k\le n}d(f^i(x_j),f^i(x_k))>r\}
\] 
and
\[
T_f(x_1,x_2,\dots,x_n;r)=\{i\in\mathbb{N}_0\colon\max_{1\le j<k\le n}d(f^i(x_j),f^i(x_k))<r\}
\]
for all $r>0$.

\begin{lem}
Let $Y$ be a non-empty subset of $X$. For any $\delta>0$,
\[
\{(x_1,x_2,\dots,x_n)\in Y^n\colon S_f(x_1,x_2,\dots,x_n;\delta)\in\mathcal{F}\}
\]
is a $G_\delta$-subset of $Y^n$ for each $\mathcal{F}\in\{\mathcal{F}_{\rm ud1},\mathcal{F}_{\rm iap}^\ast,\mathcal{F}_{\rm infinite}\}$. Moreover,
\[
\bigcap_{\epsilon>0}\{(x_1,x_2,\dots,x_n)\in Y^n\colon T_f(x_1,x_2,\dots,x_n;\epsilon)\in\mathcal{F}\}
\]
is a $G_\delta$-subset of $Y^n$ for each $\mathcal{F}\in\{\mathcal{F}_{\rm ud1},\mathcal{F}_{\rm iap}^\ast,\mathcal{F}_{\rm infinite}\}$.
\end{lem}

\begin{proof}
The first claim follows from the following equations.

\begin{align*}
\bullet\:\:&\{(x_1,x_2,\dots,x_n)\in Y^n\colon S_f(x_1,x_2,\dots,x_n;\delta)\in\mathcal{F}_{\rm ud1}\}=\bigcap_{p\ge1}\bigcap_{q\ge1}\bigcup_{m\ge q}\\&\{(x_1,x_2,\dots,x_n)\in Y^n\colon\frac{1}{m}|\{0\le i\le m-1\colon\min_{1\le j<k\le n}d(f^i(x_j),f^i(x_k))>\delta\}|>1-\frac{1}{p}\}.\\
\bullet\:\:&\{(x_1,x_2,\dots,x_n)\in Y^n\colon S_f(x_1,x_2,\dots,x_n;\delta)\in\mathcal{F}_{\rm iap}^\ast\}=\bigcap_{p\ge0}\bigcap_{m\ge1}\bigcup_{q\ge 0}\\&\{(x_1,x_2,\dots,x_n)\in Y^n\colon\min_{1\le j<k\le n}d(f^{p+qm}(x_j),f^{p+qm}(x_k))>\delta\}.\\
\bullet\:\:&\{(x_1,x_2,\dots,x_n)\in Y^n\colon S_f(x_1,x_2,\dots,x_n;\delta)\in\mathcal{F}_{\rm infinite}\}=\bigcap_{q\ge0}\bigcup_{i\ge q}\\&\{(x_1,x_2,\dots,x_n)\in Y^n\colon\min_{1\le j<k\le n}d(f^i(x_j),f^i(x_k))>\delta\}.
\end{align*}

The second claim follows from the following equations. 

\begin{align*}
\bullet\:\:&\bigcap_{\epsilon>0}\{(x_1,x_2,\dots,x_n)\in Y^n\colon T_f(x_1,x_2,\dots,x_n;\epsilon)\in\mathcal{F}_{\rm ud1}\}=\bigcap_{l\ge1}\bigcap_{p\ge1}\bigcap_{q\ge1}\bigcup_{m\ge q}\\
&\{(x_1,x_2,\dots,x_n)\in Y^n\colon\frac{1}{m}|\{0\le i\le m-1\colon\max_{1\le j<k\le n}d(f^i(x_j),f^i(x_k))<\frac{1}{l}\}|>1-\frac{1}{p}\}.\\
\bullet\:\:&\bigcap_{\epsilon>0}\{(x_1,x_2,\dots,x_n)\in Y^n\colon T_f(x_1,x_2,\dots,x_n;\epsilon)\in\mathcal{F}_{\rm iap}^\ast\}=\bigcap_{l\ge1}\bigcap_{p\ge0}\bigcap_{m\ge1}\bigcup_{q\ge 0}\\&\{(x_1,x_2,\dots,x_n)\in Y^n\colon\max_{1\le j<k\le n}d(f^{p+qm}(x_j),f^{p+qm}(x_k))<\frac{1}{l}\}.\\
\bullet\:\:&\bigcap_{\epsilon>0}\{(x_1,x_2,\dots,x_n)\in Y^n\colon T_f(x_1,x_2,\dots,x_n;\epsilon)\in\mathcal{F}_{\rm infinite}\}=\bigcap_{l\ge1}\bigcap_{q\ge0}\bigcup_{i\ge q}\\&\{(x_1,x_2,\dots,x_n)\in Y^n\colon\max_{1\le j<k\le n}d(f^i(x_j),f^i(x_k))<\frac{1}{l}\}.
\end{align*}

This completes the proof of the lemma.
\end{proof}

We recall a simplified version of a theorem of Mycielski \cite{M}. A topological space $Z$ is said to be {\em perfect} if $Z$ has no isolated point. A subset $S$ of a complete metric space $Z$ is said to be a {\em Mycielski set} if $S$ is a union of countably many Cantor sets (see \cite{BGKM}). Note that for any Mycielski set $S$ in $Z$ and an open subset $U$ of $Z$ with $S\cap U\ne\emptyset$, $S\cap U$ is an uncountable set.  

\begin{lem}[Mycielski]
Let $Z$ be a perfect complete metric space. Given any $n\ge2$ and a residual subset $R_n$ of $Z^n$, there is a Mycielski set $S$ which is dense in $Z$ and satisfies $(x_1,x_2,\dots,x_n)\in R_n$ for all distinct $x_1,x_2,\dots,x_n\in S$. Moreover, if $R_n$ is a residual subset of $Z^n$ for each $n\ge2$, then there is a Mycielski set $S$ which is dense in $Z$ and satisfies $(x_1,x_2,\dots,x_n)\in R_n$ for all $n\ge2$ and distinct $x_1,x_2,\dots,x_n\in S$.
\end{lem}

We use the following fact (see, e.g.\:Theorem 24.12 of \cite{W}).

\begin{lem}[Alexandroff]
A $G_\delta$-subset of a complete metric space is completely metrizable. 
\end{lem}

Let $Y$ be a $G_\delta$-subset of $X$. For any Furstenberg families $\mathcal{F},\mathcal{G}\subset2^{\mathbb{N}_0}$ and $n\ge2$, if
\[
\{(x_1,x_2,\dots,x_n)\in Y^n\colon S_f(x_1,x_2,\dots,x_n;\delta)\in\mathcal{F}\}
\]
is a dense $G_{\delta}$-subset of $Y^n$ for some $\delta>0$, and
\[
\bigcap_{\epsilon>0}\{(x_1,x_2,\dots,x_n)\in Y^n\colon T_f(x_1,x_2,\dots,x_n;\epsilon)\in\mathcal{G}\}
\]
is a dense $G_\delta$-subset of $Y^n$, then by Lemmas 3.2 and 3.3, there is a Mycielski set $S$ which is dense in $Y$ and satisfies
\begin{align*}
(y_1,y_2,\dots,y_n)\in&\{(x_1,x_2,\dots,x_n)\in Y^n\colon S_f(x_1,x_2,\dots,x_n;\delta)\in\mathcal{F}\}\\&\cap\bigcap_{\epsilon>0}\{(x_1,x_2,\dots,x_n)\in Y^n\colon T_f(x_1,x_2,\dots,x_n;\epsilon)\in\mathcal{G}\}
\end{align*}
for any distinct $y_1,y_2,\dots,y_n\in S$. When $\mathcal{F}=\mathcal{G}=\mathcal{F}_{\rm ud1}$ (resp.\:$\mathcal{F}=\mathcal{G}=\mathcal{F}_{\rm infinite}$), this implies that $f$ exhibits {\em distributional $n$-chaos of type 1} \cite{LO,TF} (resp.\:{\em Li--Yorke $n$-chaos} \cite{X}).

\section{S-limit shadowing}

In this section, we define the s-limit shadowing and prove a lemma for the proofs of the main theorems.

\begin{defi}
\normalfont
Let $f\colon X\to X$ be a continuous map and let $\xi=(x_i)_{i\ge0}$ be a sequence of points in $X$. For $\delta>0$, $\xi$ is called a {\em $\delta$-limit-pseudo orbit} of $f$ if $d(f(x_i),x_{i+1})\le\delta$ for all $i\ge0$, and
\[
\lim_{i\to\infty}d(f(x_i),x_{i+1})=0.
\]
For $\epsilon>0$, $\xi$ is said to be {\em $\epsilon$-limit shadowed} by $x\in X$ if $d(f^i(x),x_i)\leq \epsilon$ for all $i\ge 0$, and
\[
\lim_{i\to\infty}d(f^i(x),x_i)=0.
\]
We say that $f$ has the {\em s-limit shadowing property} if for any $\epsilon>0$, there is $\delta>0$ such that every $\delta$-limit-pseudo orbit of $f$ is $\epsilon$-limit shadowed by some point of $X$.
\end{defi}

\begin{rem}
\normalfont
 Let $f\colon X\to X$ be a continuous map and let $\xi=(x_i)_{i\ge0}$ be a sequence of points in $X$.
\begin{itemize}
\item For $\delta>0$, $\xi$ is called a {\em $\delta$-pseudo orbit} of $f$ if $d(f(x_i),x_{i+1})\le\delta$ for all $i\ge0$. For $\epsilon>0$, $\xi$ is said to be {\em $\epsilon$-shadowed} by $x\in X$ if $d(f^i(x),x_i)\leq \epsilon$ for all $i\ge 0$. We say that $f$ has the {\em shadowing property} if for any $\epsilon>0$, there is $\delta>0$ such that every $\delta$-pseudo orbit of $f$ is $\epsilon$-shadowed by some point of $X$.
\item $\xi$ is called a {\em limit-pseudo orbit} of $f$ if
\[
\lim_{i\to\infty}d(f(x_i),x_{i+1})=0,
\]
and said to be {\em limit shadowed} by $x\in X$
if
\[
\lim_{i\to\infty}d(f^i(x),x_i)=0.
\]
We say that $f$ has the {\em limit shadowing property} if every limit-pseudo orbit of $f$ is limit shadowed by some point of $X$.
\end{itemize}
It is not difficult to show that if $f$ has the s-limit shadowing property, then $f$ satisfies the shadowing property. It is also known that if $f$ has the s-limit shadowing property, then $f$ satisfies the limit shadowing property (see \cite{BGO}).  
\end{rem}

\begin{rem}
\normalfont
In \cite{MO}, it is proved that the s-limit shadowing is dense in the space of all continuous self-maps of a compact topological manifold possibly with boundary. In a recent paper \cite{BCOT}, it is shown that the s-limit shadowing is generic in the space of all continuous circle maps.
\end{rem}

\begin{lem}
Let $f\colon X\to X$ be a continuous map. Let $C\in\mathcal{C}(f)$ and $D\in\mathcal{D}(C)$. If $f$ has the s-limit shadowing property, then for any $x,y\in V^s(D)$ and $\epsilon>0$, there is $z\in V^s(D)$ such that $d(y,z)\le\epsilon$ and
\[
\lim_{i\to\infty}d(f^i(x),f^i(z))=0.
\]
\end{lem}

\begin{proof}
Since $f$ has the s-limit shadowing property, there is $\delta>0$ such that every $2\delta$-limit-pseudo orbit of $f$ is $\epsilon$-limit shadowed by some point of $X$. For this $\delta>0$, we take $E\in\mathcal{D}(C,\delta)$ with $D\subset E$. Let $m=|\mathcal{D}(C,\delta)|$ and note that $f^m(E)=E$. By property (P4) of $\sim_{C,\delta}$, we can choose $N>0$ such that for any $p,q\in E$ and $n\ge N$, there is a $\delta$-chain $(x_i)_{i=0}^{mn}$ of $f$ with $x_0=p$ and $x_{mn}=q$. Since $x,y\in V^s(D)$, we have
\[
\lim_{i\to\infty}d(f^i(x),f^i(E))=\lim_{i\to\infty}d(f^i(y),f^i(E))=0;
\]
therefore, $d(f^{2mn}(x),E)\le\delta$ and $d(f^{mn}(y),E)\le2\delta$ for some $n\ge N$. Take $x',y'\in E$ such that 
$d(f^{2mn}(x),x')\le\delta$ and $d(f^{mn}(y),y')\le2\delta$. Since $n\ge N$, the choice of $N$ gives a $\delta$-chain $(x_i)_{i=0}^{mn}$ of $f$ with $x_0=y'$ and $x_{mn}=x'$. Consider the $2\delta$-limit-pseudo orbit
\[
\xi=(z_i)_{i\ge0}=(y,f(y),\dots,f^{mn-1}(y),x_0,x_1,\dots,x_{mn-1},f^{2mn}(x),f^{2mn+1}(x),\dots),
\]
which is $\epsilon$-limit shadowed by $z\in X$. Note that $d(y,z)=d(z,z_0)\le\epsilon$. By
\[
\lim_{i\to\infty}d(f^i(x),f^i(z))=\lim_{i\to\infty}d(f^i(z),z_i)=0
\]
and $x\in V^s(D)$, we obtain $z\in V^s(D)$, completing the proof of the lemma.
\end{proof}

\begin{cor}
Let $f\colon X\to X$ be a continuous map. Let $C\in\mathcal{C}(f)$ and $D\in\mathcal{D}(C)$. If $f$ has the s-limit shadowing property, then
\[
A_n=\bigcap_{\epsilon>0}\{(x_1,x_2,\dots,x_n)\in[V^s(D)]^n\colon T_f(x_1,x_2,\dots,x_n;\epsilon)\in\mathcal{F}_{\rm ud1}\}
\]
is a dense $G_\delta$-subset of $[V^s(D)]^n$ for each $n\ge2$.
\end{cor}

\begin{proof}
Fix $p\in V^s(D)$ and $n\ge2$. Then, by Lemma 4.1,
\[
B_n=\{(x_1,x_2,\dots,x_n)\in[V^s(D)]^n\colon\lim_{i\to\infty}d(f^i(p),f^i(x_j))=0\:\:\text{for all $1\le j\le n$}\}
\]
is dense in $[V^s(D)]^n$. Since $B_n\subset A_n$, we conclude that $A_n$ is a dense $G_\delta$-subset of $[V^s(D)]^n$, proving the corollary.
\end{proof}

\section{Proofs of theorems}

In this section, we prove the main theorems. First, we prove Theorem 1.1. Let $f\colon X\to X$ be a continuous map. For any $n\ge2$ and $\delta_n>0$, we say that $(y_1,y_2,\dots,y_n)\in X^n$ is a $\delta_n$-distal $n$-tuple for $f$ if
\[
\inf_{i\ge0}\min_{1\le j<k\le n}d(f^i(y_j),f^i(y_k))>\delta_n. 
\]

\begin{proof}[Proof of Theorem 1.1]
First, the implication $(1)\implies(\text{1'})$ is proved in Lemma 2.5 of \cite{K3}. We shall prove the implication $(\text{1'})\implies(2)$. Let $D\in\mathcal{D}(C)$ and take a $\delta_n$-distal $n$-tuple $(y_1,y_2,\dots,y_n)\in D^n$ for $f$. Letting
\[
c_n=\inf_{i\ge0}\min_{1\le j<k\le n}d(f^i(y_j),f^i(y_k)),
\]
we have $0<\delta_n<c_n$. Given $a_1,a_2,\dots,a_n\in V^s(D)$ and $\epsilon>0$, by Lemma 4.1, there are $w_1,w_2,\dots,w_n\in V^s(D)$ such that $d(a_l,w_l)\le\epsilon$ and
\[
\lim_{i\to\infty}d(f^i(y_l),f^i(w_l))=0
\]
for all $1\le l\le n$. Then, we have
\[
S_f(w_1,w_2,\dots,w_n;\delta_n)\in\mathcal{F}_{\rm ud1}.
\]
Since $a_1,a_2,\dots,a_n\in V^s(D)$ and $\epsilon>0$ are arbitrary,
\[
\{(x_1,x_2,\dots,x_n)\in[V^s(D)]^n\colon S_f(x_1,x_2,\dots,x_n;\delta_n)\in\mathcal{F}_{\rm ud1}\}
\]
is a dense $G_\delta$-subset of $[V^s(D)]^n$. On the other hand, by Corollary 4.1, 
\[
\bigcap_{\epsilon>0}\{(x_1,x_2,\dots,x_n)\in[V^s(D)]^n\colon T_f(x_1,x_2,\dots,x_n;\epsilon)\in\mathcal{F}_{\rm ud1}\}
\]
is a dense $G_\delta$-subset of $[V^s(D)]^n$. Since $D\in\mathcal{D}(C)$ is arbitrary, the implication $(\text{1'})\implies(2)$ has been proved. Since
\[
\mathcal{F}_{\rm ud1}\subset\mathcal{F}_{\rm thick}\subset\mathcal{F}_{\rm infinite},
\]
the implication $(2)\implies(3)$ is obvious from the definitions. It remains to prove the implication $(3)\implies(1)$. Assume that
$(y_1,y_2,\dots,y_n)\in[W^s(C)]^n$ and $\delta_n>0$ satisfy
\[
S_f(y_1,y_2,\dots,y_n;\delta_n)\in\mathcal{F}_{\rm thick},
\]
and
\[
T_f(y_1,y_2,\dots,y_n;\epsilon)\in\mathcal{F}_{\rm infinite}
\]
for all $\epsilon>0$. By the former condition, we obtain two sequences $0\le i_1<i_2<\cdots$, $1\le j_1<j_2<\cdots$, and $x_1,x_2,\dots,x_n\in C$ such that
\[
\min_{1\le j<k\le n}d(f^i(y_j),f^i(y_k))>\delta_n
\]
for all $m\ge1$ and $i_m\le i\le i_m+j_m-1$, and
\[
\lim_{m\to\infty}f^{i_m}(y_l)=x_l
\]
for each $1\le l\le n$. Then, we have
\[
\inf_{i\ge0}\min_{1\le j<k\le n}d(f^i(x_j),f^i(x_k))\ge\delta_n>0.
\]
On the other hand, the latter condition implies
\[
\liminf_{i\to\infty}\max_{1\le j<k\le n}d(f^i(y_j),f^i(y_k))=0
\]
and so, as shown in the proof of Lemma 2.2,
\[
(y_1,y_2,\dots,y_n)\in[V^s(D)]^n
\]
for some $D\in\mathcal{D}(C)$. Given any $\delta>0$ and $E_\delta\in\mathcal{D}(C,\delta)$ with $D\subset E_\delta$, we obtain 
\[
\lim_{m\to\infty}d(f^{i_m}(y_l),f^{i_m}(E_\delta))=0
\]
for every $1\le l\le n$, and so 
\[
x_1,x_2,\dots,x_n\in E'_\delta
\]
for some $E'_\delta\in\mathcal{D}(C,\delta)$. Since $\delta>0$ is arbitrary, it follows that $x_1,x_2,\dots,x_n\in D'$ for some $D'\in\mathcal{D}(C)$, thus the implication $(3)\implies(1)$ has been proved. This completes the proof of Theorem 1.1.
\end{proof}

Next, we prove Theorem 1.2.

\begin{proof}[Proof of Theorem 1.2]
First, we prove the implication $(1)\implies(2)$. Fix any $0<\delta_n<\Delta_n$ and  $D\in\mathcal{D}(C)$. Let $p,m$ be integers with $p\ge0$ and $m\ge1$. Given $a_1,a_2,\dots,a_n\in V^s(D)$ and $\epsilon>0$ with $\delta_n+2\epsilon<\Delta_n$, since $f$ has the s-limit shadowing property, there is $\delta>0$ such that every $\delta$-limit-pseudo orbit of $f$ is $\epsilon$-limit shadowed by some point of $X$. Note that $f^p(D)\in\mathcal{D}(C)$ and take $E\in\mathcal{D}(C,\delta)$ with $f^p(D)\subset E$. Letting $\mu=|\mathcal{D}(C,\delta)|$, we have $f^\mu(E)=E$, and by property (P4) of $\sim_{C,\delta}$, we can choose $N>0$ such that for any $a,b\in E$ and $c\ge N$, there is a $\delta$-chain $(z_i)_{i=0}^{\mu c}$ of $f$ with $z_0=a$ and $z_{\mu c}=b$. Take $E'\in\mathcal{D}(C,\delta)$ with $D\subset E'$ and note that $f^p(E')=E$. Since $a_1,a_2,\dots,a_n\in V^s(D)$, by
\[
\lim_{i\to\infty}d(f^i(a_l),f^i(E'))=\lim_{i\to\infty}d(f^{p+i}(a_l),f^{p+i}(E'))=\lim_{i\to\infty}d(f^{p+i}(a_l),f^i(E))=0,
\]
we obtain $\lambda\ge N$ with $d(f^{p+\mu\lambda m}(a_l),E)\le\delta$ for all $1\le l\le n$. Fix any $1\le l\le n$ and take $a'_l\in E$ with $d(f^{p+\mu\lambda m}(a_l),a'_l)\le\delta$. Note that $f^{p+2\mu\lambda m}(D)\in\mathcal{D}(C)$ and $f^{p+2\mu\lambda m}(D)\subset E$. By the definition of $\Delta_n$, there are $y_1,y_2,\dots,y_n\in f^{p+2\mu\lambda m}(D)$ such that
\[
\min_{1\le j<k\le n}d(y_j,y_k)\ge\Delta_n.
\]
Then, the choice of $N$ gives a $\delta$-chain $(z_i^{(l)})_{i=0}^{\mu\lambda m}$ of $f$ with $z_0^{(l)}=a'_l$ and $z_{\mu\lambda m}^{(l)}=y_l$. Consider the $\delta$-limit-pseudo orbit
\[
\xi^{(l)}=(w_i^{(l)})_{i\ge0}=(a_l,f(a_l),\dots,f^{p+\mu\lambda m-1}(a_l),z_0^{(l)},z_1^{(l)},\dots,z_{\mu\lambda m-1}^{(l)},y_l,f(y_l),f^2(y_l),\dots),
\]
which is $\epsilon$-limit shadowed by $w_l\in X$. Note that
\[
d(a_l,w_l)=d(w_l,w_0^{(l)})\le\epsilon.
\]
By $y_l\in f^{p+2\mu\lambda m}(D)$ and
\[
\lim_{i\to\infty}d(f^{p+2\mu\lambda m+i}(w_l),f^i(y_l))=\lim_{i\to\infty}d(f^i(w_l),w_i^{(l)})=0,
\]
we obtain
\[
\lim_{i\to\infty}d(f^i(w_l),f^i(D))=0
\]
and so $w_l\in V^s(D)$. We also obtain
\[
d(f^{p+2\mu\lambda m}(w_l),y_l)=d(f^{p+2\mu\lambda m}(w_l),w_{p+2\mu\lambda m}^{(l)})\le\epsilon,
\]
thus
\[
\min_{1\le j<k\le n}d(f^{p+2\mu\lambda m}(w_j),f^{p+2\mu\lambda m}(w_k))\ge\Delta_n-2\epsilon>\delta_n.
\]
Since $a_1,a_2,\dots,a_n\in V^s(D)$ and $\epsilon>0$ with $\delta_n+2\epsilon<\Delta_n$ are arbitrary, it follows that
\[
\bigcup_{q\ge 0}\{(x_1,x_2,\dots,x_n)\in[V^s(D)]^n\colon\min_{1\le j<k\le n}d(f^{p+qm}(x_j),f^{p+qm}(x_k))>\delta_n\}
\]
is a dense open subset of $[V^s(D)]^n$. Since $p,m$ with $p\ge0$ and $m\ge1$ are arbitrary,
\begin{align*}
&\{(x_1,x_2,\dots,x_n)\in[V^s(D)]^n\colon S_f(x_1,x_2,\dots,x_n;\delta_n)\in\mathcal{F}_{\rm iap}^\ast\}=\bigcap_{p\ge0}\bigcap_{m\ge1}\bigcup_{q\ge 0}\\&\{(x_1,x_2,\dots,x_n)\in[V^s(D)]^n\colon\min_{1\le j<k\le n}d(f^{p+qm}(x_j),f^{p+qm}(x_k))>\delta_n\}
\end{align*}
is a dense $G_\delta$-subset of $[V^s(D)]^n$. On the other hand, by Corollary 4.1,
\[
\bigcap_{\epsilon>0}\{(x_1,x_2,\dots,x_n)\in[V^s(D)]^n\colon T_f(x_1,x_2,\dots,x_n;\epsilon)\in\mathcal{F}_{\rm ud1}\}
\]
is a dense $G_\delta$-subset of $[V^s(D)]^n$. Since $D\in\mathcal{D}(C)$ is arbitrary, the implication $(1)\implies(2)$ has been proved. Since $\mathcal{F}_{\rm ud1}\subset\mathcal{F}_{\rm infinite}$, the implication $(2)\implies(3)$ is obvious from the definitions. It remains to prove the implication $(3)\implies(1)$. Assume that
\[
\Delta_n=\inf_{D\in\mathcal{D}(C)}\sup_{x_1,x_2,\dots,x_n\in D}\min_{1\le j<k\le n}d(x_j,x_k)=0.
\]
Given any $(y_1,y_2,\dots,y_n)\in[W^s(C)]^n$, if
\[
T_f(y_1,y_2,\dots,y_n;\epsilon)\in\mathcal{F}_{\rm infinite}
\]
for all $\epsilon>0$, then
\[
\liminf_{i\to\infty}\max_{1\le j<k\le n}d(f^i(y_j),f^i(y_k))=0;
\]
therefore, as shown in the proof of Lemma 2.2,
\[
(y_1,y_2,\dots,y_n)\in[V^s(D)]^n
\]
for some $D\in\mathcal{D}(C)$. It suffices to show that
\[
S_f(y_1,y_2,\dots,y_n;\delta_n)\not\in\mathcal{F}_{\rm iap}^\ast
\]
for all $\delta_n>0$. For any $\delta_n>0$, fix $0<h<\delta_n/2$. Then, by the assumption, we have
\[
\sup_{x_1,x_2,\dots,x_n\in D_h}\min_{1\le j<k\le n}d(x_j,x_k)<h
\]
for some $D_h\in\mathcal{D}(C)$. If $\beta_h>0$ is sufficiently small, then $E_h\in\mathcal{D}(C,\beta_h)$ with $D_h\subset E_h$ satisfies
\[
\sup_{x_1,x_2,\dots,x_n\in E_h}\min_{1\le j<k\le n}d(x_j,x_k)<2h.
\]
Let $m_h=|\mathcal{D}(C,\beta_h)|$ and note that $f^{m_h}(E_h)=E_h$. By taking $p\ge0$ with $f^p(D)\subset E_h$ and  $E\in\mathcal{D}(C,\beta_h)$ with $D\subset E$, we obtain $f^p(E)=E_h$. Since $(y_1,y_2,\dots,y_n)\in[V^s(D)]^n$, it follows that
\[
\lim_{i\to\infty}d(f^i(y_l),f^i(E))=\lim_{q\to\infty}d(f^{p+m_hq}(y_l),f^{p+m_hq}(E))=\lim_{q\to\infty}d(f^{p+mq}(y_l),E_h)=0
\]
for all $1\le l\le n$; therefore, there is $Q\ge0$ such that
\[
\min_{1\le j<k\le n}d(f^{p+m_hq}(y_j),f^{p+m_hq}(y_k))<2h
\]
for every $q\ge Q$. Letting $p'=p+m_hQ$, we obtain
\[
\min_{1\le j<k\le n}d(f^{p'+m_hq}(y_j),f^{p'+m_hq}(y_k))<2h<\delta_n
\]
for every $q\ge0$, implying
\[
S_f(y_1,y_2,\dots,y_n;\delta_n)\not\in\mathcal{F}_{\rm iap}^\ast.
\]
Since $\delta_n>0$ is arbitrary, the implication $(3)\implies(1)$ has been proved. This completes the proof of Theorem 1.2.
\end{proof}

Finally, we prove Theorem 1.3.

\begin{proof}[Proof of Theorem 1.3]
First, we prove the implication $(1)\implies(2)$. Fix any $D_n\in\mathcal{D}(C)$ with $|D_n|\ge n$. We take distinct $y_1,y_2,\dots,y_n\in D_n$ and let
\[
\Gamma_n=\min_{1\le j<k\le n}d(y_j,y_k)>\delta_n>0.
\]
Let $D\in\mathcal{D}(C)$ and $q\ge0$. Given $a_1,a_2,\dots,a_n\in V^s(D)$ and $\epsilon>0$ with $\delta_n+2\epsilon<\Gamma_n$, since $f$ has the s-limit shadowing property, there is $\delta>0$ such that every $\delta$-limit-pseudo orbit of $f$ is $\epsilon$-limit shadowed by some point of $X$. Take $E\in\mathcal{D}(C,\delta)$ with $D_n\subset E$. Letting $m=|\mathcal{D}(C,\delta)|$, we have $f^m(E)=E$, and by property (P4) of $\sim_{C,\delta}$, we can choose $N>0$ such that for any $a,b\in E$ and $c\ge N$, there is a $\delta$-chain $(z_i)_{i=0}^{mc}$ of $f$ with $z_0=a$ and $z_{mc}=b$. By taking $p\ge0$ with $f^p(D)\subset E$ and $E'\in\mathcal{D}(C,\delta)$ with $D\subset E'$, we have $f^p(E')=E$. Since $a_1,a_2,\dots,a_n\in V^s(D)$, by
\[
\lim_{i\to\infty}d(f^i(a_l),f^i(E'))=\lim_{i\to\infty}d(f^{p+i}(a_l),f^{p+i}(E'))=\lim_{i\to\infty}d(f^{p+i}(a_l),f^i(E))=0,
\]
we obtain $c\ge N$ with $p+2mc\ge q$ and $d(f^{p+mc}(a_l),E)\le\delta$ for all $1\le l\le n$. Note that $f^{p+3mc}(D)\subset E$ and fix some $v\in f^{p+3mc}(D)$. Fix any $1\le l\le n$ and take $a'_l\in E$ with $d(f^{p+mc}(a_l),a'_l)\le\delta$. Then, the choice of $N$ gives a $\delta$-chain $(z_i^{(l)})_{i=0}^{2mc}$ of $f$ with $z_0^{(l)}=a'_l$, $z_{mc}^{(l)}=y_l$, and $z_{2mc}^{(l)}=v$. Consider the $\delta$-limit-pseudo orbit
\[
\xi^{(l)}=(w_i^{(l)})_{i\ge0}=(a_l,f(a_l),\dots,f^{p+mc-1}(a_l),z_0^{(l)},z_1^{(l)},\dots,z_{2mc-1}^{(l)},v,f(v),f^2(v),\dots),
\]
which is $\epsilon$-limit shadowed by $w_l\in X$. Note that
\[
d(a_l,w_l)=d(w_l,w_0^{(l)})\le\epsilon.
\]
By $v\in f^{p+3mc}(D)$ and
\[
\lim_{i\to\infty}d(f^{p+3mc+i}(w_l),f^i(v))=\lim_{i\to\infty}d(f^i(w_l),w_i^{(l)})=0,
\]
we obtain
\[
\lim_{i\to\infty}d(f^i(w_l),f^i(D))=0
\]
and so $w_l\in V^s(D)$. We also obtain
\[
d(f^{p+2mc}(w_l),y_l)=d(f^{p+2mc}(w_l),z_{mc}^{(l)})=d(f^{p+2mc}(w_l),w_{p+2mc}^{(l)})\le\epsilon,
\]
thus
\[
\min_{1\le j<k\le n}d(f^{p+2mc}(w_j),f^{p+2mc}(w_k))\ge\Gamma_n-2\epsilon>\delta_n.
\]
Recall that $p+2mc\ge q$. Since $a_1,a_2,\dots,a_n\in V^s(D)$ and $\epsilon>0$ with $\delta_n+2\epsilon<\Gamma_n$ are arbitrary, it follows that
\[
\bigcup_{i\ge q}\{(x_1,x_2,\dots,x_n)\in[V^s(D)]^n\colon\min_{1\le j<k\le n}d(f^i(x_j),f^i(x_k))>\delta_n\}
\]
is a dense open subset of $[V^s(D)]^n$. Since $q\ge0$ is arbitrary,
 \begin{align*}
&\{(x_1,x_2,\dots,x_n)\in[V^s(D)]^n\colon S_f(x_1,x_2,\dots,x_n;\delta_n)\in\mathcal{F}_{\rm infinite}\}=\bigcap_{q\ge0}\bigcup_{i\ge q}\\&\{(x_1,x_2,\dots,x_n)\in[V^s(D)]^n\colon\min_{1\le j<k\le n}d(f^i(x_j),f^i(x_k))>\delta_n\}
\end{align*}
is a dense $G_\delta$-subset of $[V^s(D)]^n$. On the other hand, by Corollary 4.1,
\[
\bigcap_{\epsilon>0}\{(x_1,x_2,\dots,x_n)\in[V^s(D)]^n\colon T_f(x_1,x_2,\dots,x_n;\epsilon)\in\mathcal{F}_{\rm ud1}\}
\]
is a dense $G_\delta$-subset of $[V^s(D)]^n$. Since $D\in\mathcal{D}(C)$ is arbitrary, the implication $(1)\implies(2)$ has been proved. Since $\mathcal{F}_{\rm ud1}\subset\mathcal{F}_{\rm infinite}$, the implication $(2)\implies(3)$ is obvious from the definitions. It remains to prove the implication $(3)\implies(1)$. Assume that
$(y_1,y_2,\dots,y_n)\in[W^s(C)]^n$ and $\delta_n>0$ satisfy
\[
S_f(y_1,y_2,\dots,y_n;\delta_n)\in\mathcal{F}_{\rm infinite},
\]
and
\[
T_f(y_1,y_2,\dots,y_n;\epsilon)\in\mathcal{F}_{\rm infinite}
\]
for all $\epsilon>0$. By the former condition, we obtain a sequence $0\le i_1<i_2<\cdots$ and $x_1,x_2,\dots,x_n\in C$ such that
\[
\min_{1\le j<k\le n}d(f^{i_m}(y_j),f^{i_m}(y_k))>\delta_n
\]
for all $m\ge1$, and
\[
\lim_{m\to\infty}f^{i_m}(y_l)=x_l
\]
for each $1\le l\le n$. Then, we have
\[
\min_{1\le j<k\le n}d(x_j,x_k)\ge\delta_n>0;
\]
therefore, $x_1,x_2,\dots,x_n$ are distinct. On the other hand, the latter condition implies
\[
\liminf_{i\to\infty}\max_{1\le j<k\le n}d(f^i(y_j),f^i(y_k))=0
\]
and so, as shown in the proof of Lemma 2.2,
\[
(y_1,y_2,\dots,y_n)\in[V^s(D)]^n
\]
for some $D\in\mathcal{D}(C)$. Given any $\delta>0$ and $E_\delta\in\mathcal{D}(C,\delta)$ with $D\subset E_\delta$, we obtain 
\[
\lim_{m\to\infty}d(f^{i_m}(y_l),f^{i_m}(E_\delta))=0
\]
for every $1\le l\le n$, and so 
\[
x_1,x_2,\dots,x_n\in E'_\delta
\]
for some $E'_\delta\in\mathcal{D}(C,\delta)$. Since $\delta>0$ is arbitrary, it follows that $x_1,x_2,\dots,x_n\in D'$ for some $D'\in\mathcal{D}(C)$, implying
\[
|D'|\ge n,
\]
thus the implication $(3)\implies(1)$ has been proved. This completes the proof of Theorem 1.3.
\end{proof}

\section{Concluding remarks}

In this section, we make some concluding remarks. Given any continuous map $f\colon X\to X$, we regard $\mathcal{C}(f)$ as the quotient space
\[
\mathcal{C}(f)=CR(f)/{\leftrightarrow}.
\]
For any $C\in\mathcal{C}(f)$, if $x\sim_C y$ implies $x=y$ for all $x,y\in C$, then $C$ is a periodic orbit or an odometer (see, e.g.\:Section 6 of \cite{K2}). Let
\[
\mathcal{C}_{\rm o}(f)=\{C\in\mathcal{C}(f)\colon\text{$x\sim_C y$ implies $x=y$ for all $x,y\in C$}\}
\]
and let
\[
\mathcal{C}_{\rm no}(f)=\mathcal{C}(f)\setminus\mathcal{C}_{\rm o}(f).
\]

If $f$ has the s-limit shadowing property, then
\[
f\colon CR(f)\to CR(f)
\]
satisfies the shadowing property. Then, for any $C\in\mathcal{C}(f)$, $x,y\in C$ with $x\sim_C y$ and $x\ne y$, and $\epsilon>0$, there are $m>0$, a closed $f^m$-invariant subset $Y$ of $X$, and a factor map
\[
\pi\colon(Y,f^m)\to(\{0,1\}^\mathbb{N},\sigma)
\]
such that
\[
Y\subset B_\epsilon(x)\cup B_\epsilon(y),
\]
where $B_\epsilon(\cdot)$ is the $\epsilon$-ball, and $\sigma\colon\{0,1\}^\mathbb{N}\to\{0,1\}^\mathbb{N}$ is the shift map (see, e.g.\:\cite{K1}). By this, we can show that for any $C\in\mathcal{C}_{\rm no}(f)$ and a neighborhood $U$ of $C$, there exists $C'\in\mathcal{C}(f)$ such that $C'$ is contained in $U$ and satisfies condition (1') of Theorem 1.1 for all $n\ge2$. In other words, letting
\begin{align*}
\mathcal{C}_\ast(f)=\{C\in\mathcal{C}(f)\colon&\text{There is a sequence $\delta_n>0$, $n\ge2$, such that for any $D\in\mathcal{D}(C)$},\\&\text{$D^n$ contains a $\delta_n$-distal $n$-tuple for $f$ for all $n\ge2$}\},
\end{align*}
we have
\[
\mathcal{C}_{\rm no}(f)\subset\overline{\mathcal{C}_\ast(f)}.
\]

For any continuous map $f\colon X\to X$, let
\[
\mathcal{C}_{\rm pte}(f)=\{C\in\mathcal{C}(f)\colon h_{\rm top}(f|_C)>0\},
\]
where $h_{\rm top}(\cdot)$ denotes the topological entropy. Then, we have
\[
\mathcal{C}_{\rm pte}(f)\subset\mathcal{C}_{\rm no}(f),
\]
and if $h_{\rm top}(f)>0$, then
\[
\mathcal{C}_{\rm pte}(f)\ne\emptyset.
\]
As a consequence, if $f$ has the s-limit shadowing property and satisfies $h_{\rm top}(f)>0$, then
\[
\mathcal{C}_\ast(f)\ne\emptyset.
\]


\begin{thebibliography}{99}

\bibitem{A} E.\:Akin, The general topology of dynamical systems. Graduate Studies in Mathematics, 1. American Mathematical Society, Providence, R.I., 1993.

\bibitem{An} D.V.\:Anosov, Geodesic flows on closed Riemann manifolds with negative curvature. Proc. Steklov Inst. Math. 90 (1967), 235 p.

\bibitem{AH} N.\:Aoki, K.\:Hiraide, Topological theory of dynamical systems. Recent advances. North--Holland Mathematical Library, 52. North--Holland Publishing Co., 1994.

\bibitem{BGO} A.D.\:Barwell, C.\:Good, P.\:Oprocha, Shadowing and expansivity in subspaces. Fund. Math. 219 (2012), 223--243.

\bibitem{BGKM} F.\:Blanchard, E.\:Glasner, S.\:Kolyada, A.\:Maass, On Li--Yorke pairs. J. Reine Angew. Math. 547 (2002), 51--68.

\bibitem{BCOT} J.\:Bobok, J.\:\v{C}in\v{c}, P.\:Oprocha, S.\:Troubetzkoy, S-limit shadowing is generic for continuous Lebesgue measure-preserving circle maps, Ergodic Theory Dynam. Systems 43 (2023), 78--98.

\bibitem{B} R.\:Bowen, Equilibrium states and the ergodic theory of Anosov diffeomorphisms. Lecture Notes in Mathematics, 470. Springer--Verlag, 1975.

\bibitem{BMR} W.R.\:Brian, J.\:Meddaugh, B.E.\:Raines, Chain transitivity and variations of the shadowing property. Ergodic Theory Dynam. Systems 35 (2015), 2044--2052.

\bibitem{C} C.\:Conley, Isolated invariant sets and the Morse index. CBMS Regional Conference Series in Mathematics, 38. American Mathematical Society, Providence, R.I., 1978.

\bibitem{F} H.\:Furstenberg, Recurrence in ergodic theory and combinatorial number theory. Princeton University Press, 1981.

\bibitem{GH} W.H.\:Gottschalk, G.A.\:Hedlund, Topological dynamics. Colloquium Publications of the American Mathematical Society, 36. American Mathematical Society, Providence, R.I., 1955.

\bibitem{H} M.\:Hurley, Lyapunov functions and attractors in arbitrary metric spaces. Proc. Amer. Math. Soc. 126 (1998), 245--256.

\bibitem{K1} N.\:Kawaguchi, Distributionally chaotic maps are $C^0$-dense. Proc. Amer. Math. Soc. 147 (2019), 5339--5348.

\bibitem{K2} N.\:Kawaguchi, Maximal chain continuous factor. Discrete Contin. Dyn. Syst. 41 (2021), 5915--5942.

\bibitem{K3} N.\:Kawaguchi, A type of shadowing and distributional chaos. Dyn. Syst. 36 (2021), 572--585.

\bibitem{LO} J.\:Li, P.\:Oprocha, On $n$-scrambled tuples and distributional chaos in a sequence. J. Differ. Equations Appl. 19 (2013), 927--941.

\bibitem{LY1} J.\:Li, X.\:Ye, Recent development of chaos theory in topological dynamics. Acta Math. Sin., Engl. Ser. 32 (2016), 83--114.

\bibitem{LY2} T.Y.\:Li, J.A.\:Yorke, Period three implies chaos. Amer. Math. Monthly 82 (1975), 985--992.

\bibitem{MO} M.\:Mazur, P.\:Oprocha, S-limit shadowing is $C^0$-dense. J. Math. Anal. Appl. 408 (2013), 465--475.

\bibitem{M} J.\:Mycielski, Independent sets in topological algebras. Fund. Math. 55 (1964), 139--147.

\bibitem{P} S.Yu.\:Pilyugin, Shadowing in dynamical systems. Lecture Notes in Mathematics, 1706. Springer--Verlag, 1999.

\bibitem{RW} D.\:Richeson, J.\:Wiseman, Chain recurrence rates and topological entropy. Topology Appl. 156 (2008), 251--261.

\bibitem{SS} B.\:Schweizer, J.\:Sm\'ital, Measures of chaos and a spectral decomposition of dynamical systems on the interval. Trans. Amer. Math. Soc. 344 (1994), 737--754.

\bibitem{S} T.\:Shimomura, On a structure of discrete dynamical systems from the view point of chain components and some applications. Japan. J. Math., New Ser. 15 (1989), 99--126.

\bibitem{TF} F.\:Tan, H.\:Fu, On distributional $n$-chaos. Acta Math. Sci., Ser. B, Engl. Ed. 34 (2014), 1473--1480. 

\bibitem{TX} F.\:Tan, J.\:Xiong, Chaos via Furstenberg family couple. Topology Appl. 156 (2009), 525--532.

\bibitem{W} S.\:Willard, General topology. Dover Publications, Inc., 2004.

\bibitem{WOC} X.\:Wu, P.\:Oprocha, G.\:Chen, On various definitions of shadowing with average error in tracing. Nonlinearity 29 (2016), 1942--1972.

\bibitem{X} J.\:Xiong, Chaos in topological transitive systems, Sci. China Ser. A 48 (2005), 929--939.

\bibitem{XLT} J.\:Xiong, J.\:L\"{u}, F.\:Tan, Furstenberg family and chaos. Sci. China Ser. A 50 (2007), 1325--1333.

\end{thebibliography}
\end{document}